%% file: paper.tex
\documentclass[hidelinks,onefignum,onetabnum]{siamart220329}
\input{macro.tex}

\usetikzlibrary{fit}
\usetikzlibrary{arrows}
\numberwithin{equation}{section}
\hypersetup{linkbordercolor=red}

\input{ex_shared}

\ifpdf
\hypersetup{
  pdfauthor={A. Blaustein}
}
\fi

\begin{document}

\maketitle

\begin{abstract}
This work tackles the diffusive limit for the
Vlasov-Poisson-Fokker-Planck model. We derive \textit{a priori}
estimates which hold without restriction on the phase-space dimension
and propose a strong convergence result in a $L^2$ space. Furthermore,
we strengthen previous results by obtaining an explicit convergence
rate arbitrarily close to the (formal) optimal rate, provided that the
initial data lies in some $L^p$ space with $p$ large enough. Our
result holds on bounded time intervals whose size grow to infinity in
the asymptotic limit with explicit lower bound. The analysis relies on identifying the right set of phase-space coordinates to study the regime of interest. In this set of coordinates the limiting model arises explicitly.
\end{abstract}



\section{Introduction}
\label{sec:1}

\subsection{Physical model and motivation}\label{sec:11}
In this article, we study a plasma composed of moving electrons and fixed ions. We denote by $f^\eps(t,\bx,\bv)\geq0$ the density
of electron at time $t\in\R^+$, position $\bx \in \K^d$ with $\K \in
\left\{\T,\R\right\}$ and velocity $\bv \in \R^d$ whereas the
density of ions is given by $\rho_i(\bx)\geq0$ with $\rho_i \in
L^1\left(\K^d\right)$. We focus on the Vlasov-Poisson-Fokker-Planck
(VPFP) model, which describes a situation where the particles interact
through electrostatic effects and where electrons are subjected to collisions with the ion background. Considering the regime in which the electron/ion mass ratio and the mean free path of electrons have the same magnitude, the VPFP model reads
\begin{equation}\label{kinetic VPFP}
	\left\{
	\begin{array}{l}
		\displaystyle
		\ds\partial_t \, f^\eps
		\,+\,
		\frac{1}{\eps}\,
		\bv\cdot
		\nabla_{\bx}\,
		f^\eps
		\,+\,
		\frac{1}{\eps}\,
		\bE^\eps\cdot
		\nabla_{\bv}\,
		f^\eps
		\,=\,
		\frac{1}{\eps^2}\,
		\nabla_{\bv}\cdot
		\left[\,
		\bv\,
		f^\eps
		\,+\,
		\nabla_{\bv}\,
		f^\eps\,
		\right]
		\,
		,\\[1.5em]
		\displaystyle
		\bE^\eps\,=\,-\,\nabla_{\bx}\,\phi^\eps
		\,,\quad
		-\,\Delta_{\bx}\,
		\phi^\eps
		\,=\,
		\rho^\eps
		-
		\rho_i\,,
		\quad
		\rho^\eps
		\,=\,
		\int_{\R^d}\,
		f^\eps\,\dD\bv\,,\\[1.5em]
		\ds f^\eps(0,\bx,\bv)
		\,=\,
		f^\eps_0(\bx,\bv)\,,
	\end{array}
	\right.
\end{equation}
where the self consistent electric field $\bE^\eps$ is induced by
Coulombian interactions between charges whereas the Fokker-Planck
operator on the right-hand side of the first line in \eqref{kinetic VPFP} accounts for collisions with the ion background. A detailed description of the re-scaling process in order to derive \eqref{kinetic VPFP} may be found in \cite{Filbet_Negulescu22,Poupaud_Soler}. Since mass is conserved along the trajectories of \eqref{kinetic VPFP}, we normalize $f^\eps$ as follows
\[
\int_{
	\K^d\times\R^d
} f^\eps_0\,\dD\bx\,\dD\bv
\,=\,
1\,.
\]
When $\K\,=\,\T$, we also impose the compatibility assumption
\[
\int_{
	\K^d\times\R^d
}f^\eps_0\,\dD\bx\,\dD\bv
\,=\,
\int_{\K^{d}} \rho_i\,\dD\bx\,,
\]
which is then satisfied for all positive time $t\geq0$. In this article, we focus on the asymptotic analysis of \eqref{kinetic VPFP} in the diffusive regime corresponding to the limit $\eps \ll 1$.
\subsection{ Formal derivation}  In this section, we carry out a formal analysis of
the diffusive regime. Our purpose is to
derive the limit of $f^\eps$ as $\eps\rightarrow 0$ and to explain the key idea of this article in order to rigorously
justify this limit. Considering the leading order term in equation
\eqref{kinetic VPFP}, we deduce that as $\eps$ vanishes, $f^\eps$ converges to the following local equilibrium of the Fokker-Planck operator
\begin{equation*}
	f^\eps
	\,
	\underset{\eps\rightarrow 0}{\sim}
	\,
	\rho^\eps\otimes\cM
	\,,
\end{equation*}
where $\cM$ stands for the standard Maxwellian distribution over $\R^d$
\begin{equation}\label{maxwellian}
	\mathcal{M}(\bv)\,=\,
	\left(2\pi
	\right)^{-\,d/2}\,\exp
	\left(
	- \,\frac{1}{2}\,|\bv|^2
	\right),
\end{equation}
and where the spatial distribution $\rho^\eps$ given in \eqref{kinetic VPFP} solves the following equation, obtained after integrating the first line in \eqref{kinetic VPFP} with respect to $\bv$
\[
\partial_t \, \rho^\eps
\,+\,\frac{1}{\eps}\,
\nabla_{\bx}\cdot
\int_{\R^d}
\bv\,
f^\eps\,\dD \bv 
\,=\,0\,.
\]
It is unfortunately trickier to compute the limit of $\rho^\eps$. The difficulty stems from the free transport operator in the first line of \eqref{kinetic VPFP} which induces a stiff dependence with respect to $f^\eps$ in the equation on $\rho^\eps$. The key idea is therefore to cancel this stiff dependence by considering a modified spatial distribution $\pi^\eps$. To this aim, we introduce a re-scaled version $g^\eps$ of $f^\eps$
\[
\left(\tau_{\eps\bv}\,g^\eps\right)(t,\bx
,\bv)
\,=\,
f^\eps(t,\bx
,\bv)\,,\quad\forall\,(t,\bx,\bv)\in\R^+\times\K^{d}\times \R^d\,,
\]
where $\ds\tau_{\bx_0}\,g$ denotes the function obtained by translating any function $g$ in the direction $\ds\bx_0\in \R^d$, that is  
\[
\tau_{\bx_0}\,g(t,\bx,\bv)\,:=\,
g(t,\bx
+\bx_0
,\bv)\,,\quad\forall\,(t,\bx,\bv)\in\R^+\times\K^{d}\times \R^d\,.
\]
Defining the new variable $\by:=\bx+\eps\,\bv$ and operating the change of variable $\bx \rightarrow \by$ in the first line of \eqref{kinetic VPFP}, it turns out that
	$g^\eps$ solves the following equation 
	\[
	\ds\partial_t  g^\eps
	\,+\,
	\left(
	\nabla_{\by}
	+
	\frac{1}{\eps}
	\nabla_{\bv}
	\right)\cdot
	\left[\,
	\tau^{-1}_{\eps\bv}
	\left(\,\bE^\eps
	\,
	f^\eps\right)
	\,\right]
	\,-\,
	\Delta_{\by}\,
	g^\eps
	\,=\,
	\frac{1}{\eps^2}\,
	\nabla_{\bv}\cdot
	\left[\,
	\bv\,
	g^\eps
	+
	\nabla_{\bv}\,
	g^\eps\,
	\right]
	\,+\,
	\frac{2}{\eps}\,
	\nabla_{\bv}\cdot\nabla_{\by}\,g^\eps,
	\]
	where one may notice that the free transport operator has been canceled in the re-scaling process. Therefore, the marginal $\pi^\eps$ of $g^\eps$ defined as
	\[
	\pi^\eps\left(t,
	\by
	\right)
	\,=\,
	\int_{\R^d}
	g^\eps
	\left(t,
	\by
	,\bv
	\right)\dD\bv
	\,=\,
	\int_{\R^d}
	\tau^{-1}_{\eps\bv}
	f^\eps\left(t,
	\by
	,\bv
	\right)\dD\bv\,,\quad\forall \left(t,\by\right) \in\R^+\times\K^d\,,
	\]
	solves the following equation, obtained after integrating the equation on $g^\eps$ with respect to $\bv$
	\begin{equation}\label{eq pi eps}
		\ds\partial_t \, \pi^\eps
		\,+\,
		\nabla_{\by}\cdot
		\left[\,
		\int_{\R^d}
		\tau^{-1}_{\eps\bv}
		\left(\,\bE^\eps\,f^\eps\right)
		\left(t,
		\by
		,\bv
		\right)
		\,
		\dD\bv
		\,\right]
		\,-\,
		\Delta_{\by}\,
		\pi^\eps
		\,=\,
		0
		\,.
\end{equation}
In the limit $\eps\,\rightarrow\,0$, we have
$
\ds
\rho^\eps\,\sim\,\pi^\eps
$ and thus $
\ds
\bE^\eps\,\sim\,\bI^\eps
$, where 
\begin{equation}\label{def I eps}
	\displaystyle
	\bI^\eps\,=\,-\,\nabla_{\by}\,\psi^\eps
	\,,
	\quad
	-\,\Delta_{\by}\,
	\psi^\eps
	\,=\,
	\pi^\eps
	\,-\,
	\rho_i
	\,.
\end{equation}
Furthermore, the translation operator $\tau^{-1}_{\eps\bv}$ in \eqref{eq pi eps} cancels as $\eps \rightarrow 0$. Therefore, the dependence with respect to $f^\eps$ is removed from the nonlinear term in \eqref{eq pi eps}, that is
\[
\int_{\R^d}
\tau^{-1}_{\eps\bv}
\left(\,\bE^\eps\,f^\eps\right)
\left(t,
\by
,\bv
\right)
\,
\dD\bv
\,\sim\,
\bI^\eps\,\pi^\eps\,.
\]
Therefore, we deduce
\begin{equation*}
	\pi^\eps
	\,
	\underset{\eps\rightarrow 0}{\longrightarrow}
	\,
	\rho
	\,,
\end{equation*}
where $\rho$ solves the following drift-diffusion-Poisson equation
\begin{equation}\label{ddP equation}
	\left\{
	\begin{array}{l}
		\displaystyle\partial_t\,\rho\,+\,
		\nabla_{\bx}\cdot
		\left[\,
		\mathbf{E}\,\rho
		\,\right]
		\,-\,
		\Delta_{\bx}\,\rho
		\,=\,0
		\,
		,\\[1.5em]
		\displaystyle
		-\,\Delta_{\bx}\,
		\phi
		\,=\,
		\rho
		\,-\,
		\rho_i\,\,,\quad
		\bE\,=\,-\,\nabla_{\bx}\,\phi
		\,,
	\end{array}
	\right.
\end{equation}
Since we have $\pi^\eps\,\sim\,\rho^\eps$, this leads to the following formal result
\begin{equation*}
	f^\eps
	\,
	\underset{\eps\rightarrow 0}{\longrightarrow}
	\,
	\rho\otimes\cM
	\,.
\end{equation*}
As it turns out, the set of coordinates $(
\by,\bv)$, with $\by:=\bx
+
\eps\,\bv$ is very convenient to study the diffusive regime: it removes from the equation on $\rho^\eps$ the stiff dependence with respect to $f^\eps$ due to the transport operator. Instead, we end up with equation \eqref{eq pi eps} which is close to the limiting model \eqref{ddP equation}. Therefore, the modified spatial distribution $\pi^\eps$ will play a key role in our analysis.

The present work consists in rigorously justifying this formal
derivation. It is the continuation of a long process in order to justify the diffusive limit of the VPFP model.\\
The story starts in \cite{Poupaud_Soler}, in which F. Poupaud and J. Soler demonstrate strong convergence of $f^\eps$ on short time intervals without restriction on the dimension. The present article is somehow close to \cite{Poupaud_Soler} since it uses the same functional framework. Their key idea is to estimate the norm $\LLL{f^\eps}_p$ (defined in Section \ref{sec:main} below), which in turn provides a $L^\infty$-control over the field $\bE^\eps$.  F. Poupaud and J. Soler manage to control $\LLL{f^\eps}_p$ on short time intervals. This estimate allow them to prove strong compactness of the sequence $\left(\rho_{\eps}\right)_{\eps>0}$ and to pass to the limit in the nonlinear term thanks to their $L^\infty$-control over $\bE^\eps$. In the same article, F. Poupaud and J. Soler remove the short time restriction and obtain local in time weak convergence in dimension $d=2$. In this case, the Coulomb potential has a particular structure which allows them to pass to the limit in the nonlinear term. Their strategy was then extended by T. Goudon \cite{Goudon05} who proved similar results for Newtonian interactions when $d=2$. \\
The story goes on in \cite{El_Ghani_Masmoudi}, in which N. El Ghani and N. Masmoudi prove strong and local in time convergence without restriction on the dimension $d$. Their method relies on averaging lemmas \cite{Golse_Lions_Perthame_Sentis88,DiPerna_Lions_Meyer91} to prove strong compactness of the spatial distributions $\left(\rho_{\eps}\right)_{\eps>0}$ associated to free energy solutions $\left(f_{\eps}\right)_{\eps>0}$ to \eqref{kinetic VPFP}. In such a weak regularity framework, the nonlinear term in \eqref{kinetic VPFP} may not be defined. Hence, authors use renormalization techniques introduced by  R. J. Diperna and P.-L. Lions \cite{DiPerna_Lions88} to pass to the limit in the nonlinear term. This approach was initially designed to treat the case of a linearized Boltzmann operator \cite{Masmoudi_Tayeb07}. Since then, it has been extended in various directions including multi-species models \cite{Wu_Lin_Liu15}, Vlasov-Maxwell-Fokker-Planck model \cite{Ghani10} and strongly magnetized plasma models \cite{Herda16}.\\
More recently, authors adapted hypocoercivity methods \cite{Villani:AMS,DMS,herau2017} to the present situation in order to achieve global in time convergence. This is the case of  \cite{Herda_Rodrigues}, in which M. Herda and M. Rodrigues prove global in time strong convergence in weakly nonlinear regime (that is, under joint restrictions on the Debye length and the size of the initial data) when $d=2$. The main difficulty to prove global in time convergence is that the Fokker-Planck operator in \eqref{kinetic VPFP} acts only on the velocity variable and therefore gives no straightforward information regarding the asymptotic behavior of $\rho_\eps$ as $t\rightarrow +\infty$. Hypocoercivity methods rely on constructing of a modified relative entropy functional designed to recover dissipation with respect to the spatial variable. In \cite{Herda_Rodrigues}, M. Herda and M. Rodrigues design such functional, allowing to deduce global in time convergence of the sequence $\left(f^\eps\right)_{\eps>0}$. Using similar methods, exponential relaxation as $t\rightarrow +\infty$ had already been proved in \cite{Herau_Thomann16}, in a weakly nonlinear setting. We also mention \cite{Lanoir/Dolbeault/Xingyu/M.Lazhar}, which proves that the linearized model associated to \eqref{kinetic VPFP} relaxes exponentially fast towards equilibrium as $t\rightarrow +\infty$ with uniform rates in the limit $\eps\rightarrow 0$. Authors deduce the same result on the nonlinear model in the case $d=1$. \\
In perturbative settings, precise results are available \cite{Zhong22}, based on a spectral analysis of \eqref{kinetic VPFP}, treating simultaneously the diffusion limit with explicit convergence rates in $\eps$ and the long time behavior with optimal exponential rate in $t$.

In this article, we propose a strong convergence result in some $L^2$ space
and prove that it occurs at rate $O(\eps^{\beta})$, with
$\beta\,=\,(p-d)/(p-1)$ if $\ds f^\eps_0$ lies in some $\ds L^p$
space: the (formal) optimal convergence rate is reached as
$p\rightarrow +\infty$. Our analysis is non-perturbative and it holds in any dimension $d\geq1$. Our convergence estimate holds on bounded time intervals $[0,T^\eps]$, for some explicitly given $T^\eps$ (see Theorem \ref{main:th}) which satisfies $T^\eps \rightarrow +\infty$ as $\eps \rightarrow 0$. We point out that it should be possible to adapt our analysis to Newtonian interactions as long as the macroscopic model does not develop singularities but we do not follow this path to avoid these issues.

The article is organized as follows: in Section \ref{sec:main}, we give our functional setting and state our main result, in Section \ref{sec:estimates} we derive uniform estimates for the solution to \eqref{kinetic VPFP}, then we conclude with Section \ref{sec:proof} in which we prove our main result.
\section{Functional setting and main result}\label{sec:main}
In the forthcoming analysis, we work with the following norm defined for all exponent $p\geq1$ 
\[
\LLL{
	f}_{p
}
\,=\,
\left(
\int_{\K^d\times\R^d}
\left|\,\frac{f}{\cM}\,\right|^p\cM
\,\dD\bx \,\dD \bv
\right)^{1/p}\,,
\]
and denote by $\ds L^p
\left(
\cM
\right)$ the set of all function whose latter norm is finite. Furthermore, we denote $\ds\left\|\cdot\right\|_{L^{p}}$ the usual norm over $L^p\left(\K^d\right)$.

Existence and uniqueness theory for \eqref{kinetic VPFP} has been
widely investigated and therefore it will not be our concern here. We mention \cite{Degond86} in which global classical solutions are constructed in dimension $d=1,2$, \cite{Kosuke01,Kosuke_Strauss00} which extend this result to dimension $d=3$ in both the friction and frictionless cases and notably prove regularizing effects and thus obtain infinite regularity, \cite{Bouchut93} in which existence and uniqueness of a global strong solution is obtained when $d=3$ and with uniform bounds on the initial data and then \cite{Bouchut95} in which regularizing effects are proved for this weak solution, finally, we mention \cite{Carrillo_Soler95} which treats the case of an initial data in $L^p$ and constructs solutions in dimensions $d=3,4$. Since this article does not require any constraint on the dimension, we consider a strong solution to \eqref{kinetic VPFP} in dimension $d\geq1$. Our main result reads as follows
\begin{theorem}\label{main:th} Consider some exponent $p>d$ and set
	\[ \left(\gamma\,,\,\beta\right)=\left(1-\frac{d}{p}\,,\,\frac{p-d}{p-1}\right).\]
	Suppose that the sequence $
	\left(
	f^\eps_0
	\right)_{\eps\,>\,0} 
	$ meets the following assumption
	\begin{equation}\label{hyp f eps 0}
		\LLL{
			f^\eps_0}_{p
		}
		+
		\sup_{|\bx_0|\,\leq\,1}\left(
		|\bx_0|^{-\beta}\,
		\LLL{
			\tau_{\bx_0}
			f^\eps_0
			-
			f^\eps_0
		}_{p
		}\right)
		+
		\eps^{-\beta}
		\left(
		\left\|
		\pi_0^\eps-
		\rho_0^\eps
		\right\|_{L^{p}
		}+
		\left\|
		\rho_0^\eps
		-
		\rho_0
		\right\|_{L^{p}
		}
		\right)
		\leq\,
		m_p\,,
	\end{equation}
	for some positive constant $m_p$ independent of $\eps$. On top of that, suppose 
	\begin{equation}\label{hyp rho 0 rho i}
		\rho_0\in L^p\cap L^\infty\left(\K^d\right)\,,\quad\textrm{and}\quad\rho_i\in L^{p+1}\cap L^\infty\left(\K^d\right)\,,
	\end{equation}
	and define
	\[
	C_{\rho_0,\rho_i}\,:=\,
	\max{
		\left(
		\|\,\rho_0\,\|^2_{L^{p}}
		\,,\,
		\|\,\rho_i\,\|_{L^{p+1}}^{2-2/p^2}
		\,,\,
		\|\,\rho_0\,\|_{L^{\infty}}
		\,,\,
		\|\,\rho_i\,\|_{L^{\infty}}
		\right)
	}\,.
	\]

	Consider strong solutions $(f^\eps)_{\eps\,>\,0}$ to \eqref{kinetic VPFP} and $\rho$ to \eqref{ddP equation} with initial data $
	\left(
	f^\eps_0
	\right)_{\eps\,>\,0} 
	$ and $\rho_0$ respectively. For all time $T\,>\,0$, there exist two positive constants $C_T$ and $\eps_T$ such that for all $\eps\,\leq\,\eps_T$, it holds
	\[
	\left\|
	f^\eps
	\,-\,
	\rho\,\cM
	\right\|_{L^2
		\left(
		[0,T]\,,\,
		L^{2}
		\left(
		\cM
		\right)
		\right)
	}
	\,\leq\,
	C_T\,\eps^{\beta}
	\,.
	\]
	
	More precisely, there exists a constant $C$ only depending on exponent $p$ and dimension $d$ such that for all $\eps \leq 1$ it holds
	\[
	\left\|
	f^\eps
	\,-\,
	\rho\,\cM
	\right\|_{L^2
		\left(
		[0,t]\,,\,
		L^{2}
		\left(
		\cM
		\right)
		\right)
	}
	\,\leq\,
	\eps^{\beta}
	\left(C\,
	m_p^{2\,p'}\,
	e^{C_{\rho_0,\rho_i}
		C\,
		t}
	\,+\,
	e^{C_{\rho_0,\rho_i}^{-1}C\, m_p^2\,e^{C_{\rho_0,\rho_i}C\,t}}
	\right),
	\]
	where $p'=p/(p-1)$ and for all time $t$ less than $T^\eps$, where $T^\eps$ is defined as
	\[
	T^\eps\,=\,
	\frac{1}{C_{\rho_0,\rho_i}
		\,
		C}\,
	\ln{
		\left(
		1\,+\,
		\frac{C_{\rho_0,\rho_i}}{
			4\,
			\LLL{f^\eps_0}_p^2+
			C\,m_p^2\,\eps^{\gamma\left(1+\frac{2}{p-1}\right)}
		}\,
		\eps^{-\gamma}
		\right)
	}\,,
	\]
	which ensures $\,\,\ds T^\eps\longrightarrow +\infty\,$ as $\,\eps\rightarrow 0$.
\end{theorem}
The main difficulty consists in estimating the nonlinear term $\bE^\eps f^\eps$ in \eqref{kinetic VPFP}. Indeed, if we follow the same method as in \cite{Poupaud_Soler} and take directly its $\LLL{\cdot}_p$-norm, we end up with the following differential inequality for $\LLL{f^\eps}_p$, which blows up in finite time
\[
\frac{\dD}{\dD t}\,\LLL{f^\eps}^p_p
\;\;\lesssim\;\;
\LLL{f^\eps}^{p+2}_p\,.
\]
The key point is therefore to include the re-scaled marginal $\pi^\eps$ in our computations. More precisely, we perform the following decomposition
\[
\bE^\eps\,=\left(\bE^\eps-\bI^\eps\right)+\left(\bI^\eps-\bE\right)+\,\bE\,,
\]
where $\bE$ and $\bI^\eps$ are given in \eqref{ddP equation} and \eqref{def I eps} respectively.
We rely on a functional inequality to prove that $\left(\bE^\eps-\bI^\eps\right)$ is of order $\eps^\gamma\,\LLL{f^\eps}_p$ and estimate $\left(\bI^\eps-\bE\right)$ and  $\bE$ thanks to the properties of drift-diffusion equation \eqref{ddP equation}. It enables to derive the following differential inequality
\[
\frac{\dD}{\dD t}\,\LLL{f^\eps}^p_p
\;\;\lesssim\;\;
\eps^{\gamma}\,\LLL{f^\eps}^{p+2}_p\,.
\]
From the latter inequality, we bound $\LLL{f^\eps}_p$ on time intervals with size of order $\left|\ln{\eps}\right|$ ; this provides a global in time estimate in the limit $\eps \rightarrow 0$.
\section{\textit{A priori} estimates}\label{sec:estimates}
The main object of this section consists in deriving \textit{a priori} estimates for \eqref{kinetic VPFP} in $L^p\left(\cM\right)$ (see \eqref{estimee norme Lp f eps} in Proposition \ref{a:priori:f:eps}). Building on this key estimate, we deduce that $\pi^\eps$ given by \eqref{eq pi eps} converges towards the solution $\rho$ to the macroscopic model \eqref{ddP equation} (see \eqref{estimee norme Lp pi eps rho} in Proposition \ref{a:priori:f:eps}) and prove equicontinuity for solutions to \eqref{kinetic VPFP} (see Proposition \ref{equicontinuity f eps}). Let us first introduce some notations and recall a functional inequality that will be used in the proofs of this section. For all function $\rho\in L^p\cap L^1\left(\K^d\right)$ with $p>1$, which in the case $\K\,=\,\T$ meets the compatibility assumption
\[
\int_{\K^d}\rho\,\dD \bx\,=\,0\,,
\]
there exists a unique solution  $\Delta^{-1}_{\bx}\,\rho$ in $W^{2,p}\left(\K^d\right)$ to the Poisson equation (see \cite[Section $9.6$]{Gilbarg_Trudinger})
\begin{equation*}
	\left\{
	\begin{array}{l}
		\displaystyle
		\Delta_{\bx}\,
		\phi
		\,=\,
		\rho\,,\\[1.5em]
		\ds \int_{\K^d}\phi\,\dD \bx
		\,=\,
		0\,,\quad\text{if}\;\;\K\,=\,\T\,.
	\end{array}
	\right.
\end{equation*}
Furthermore, thanks to Morrey's inequality, we have the following estimate
\begin{equation}\label{Morrey inequality}
	\left\|
	\nabla_{\bx}\,\Delta^{-1}_{\bx}
	\,\rho
	\right\|_{
		\scC^{0,\gamma}
	}
	\,\leq\,
	m_{d,p}
	\left\|
	\rho
	\right\|_{L^{p}}\,,
\end{equation}
for all exponent $p>d$, where the constant $m_{d,p}$ only depends on $(d,p)$ and where $\scC^{0,\gamma}$ stands for the set of bounded, $\gamma$-H\"older functions, with $\gamma
\,=\,
1
\,-\,
d/p$.\\
In this section, we denote by $
\Delta_{p}$ the dissipation of $L^p$-norms due to the Laplace operator
\[
\Delta_{p}
\left[\,
\rho
\,
\right]
\,=\,
(p-1)\,
\int_{\K^d}
\left|\,\nabla_{\bx}\,
\rho\,\right|^2
\left|\,\rho\,\right|^{p-2}
\,\dD\bx\,,
\]
and by $\cD_p$ the dissipation of $L^p\left(\cM\right)$-norms due to the Fokker Planck operator
\[
\cD_{p}
\left[\,
f
\,
\right]
\,=\,
(p-1)
\int_{\K^d\times\R^d}
\left|\,\nabla_{\bv}
\left(
\frac{f}{\cM}\right)\right|^2
\left|\,\frac{f}{\cM}\,\right|^{p-2}\cM
\,\dD\bx \,\dD \bv\,.
\]
We start with the following intermediate result which will be essential in order to propagate $L^p\left(\cM\right)$-norms. 
\begin{lemma}\label{I-E}
	Consider a smooth solution $f^\eps$ to equation \eqref{kinetic VPFP}. For all exponent $p>d$ and all positive $\eps$, it holds
	\[
	\left\|
	\,
	\bE^{\eps}
	-\,
	\bI^\eps\,
	\right\|_{L^{\infty}}
	\,\leq\,
	C_{d,p}\,
	\eps^{\gamma}\,
	\LLL{
		f^\eps}_{p
	}
	\,,
	\quad
	\forall\,t\in\R^+\,,
	\]
	where exponent $\gamma$ is given by $\gamma\,=\,1-
	d\,/\,p
	$, and where $C_{d,p}$ is a constant only depending the dimension $d$ and exponent $p$. In the latter estimates, the electric fields $\bE^\eps$ and $\bI^\eps$ are given by \eqref{kinetic VPFP} and \eqref{def I eps}.
\end{lemma}

\begin{proof}
	We consider some positive $\eps$ and some $(t,\bx) \in \R^+\times\K^{d}$. We have
	\[
	\left(
	\rho^\eps
	\,-\,
	\pi^\eps\right)(t,\bx)
	\,=\,
	\int_{\R^d}
	\left(f^\eps
	\,-\,
	\tau_{\eps\,\bv}^{-1}f^\eps\right)(t,\bx,\bv)
	\,\dD\bv
	\,.
	\]
	Applying the operator $\ds\nabla_{\bx}\,\Delta^{-1}_{\bx}$ to the latter relation and taking the supremum over all $\bx$ in $\K^{d}$, we obtain
	\[
	\left\|\,
	\bE^{\eps}-\,
	\bI^\eps\,
	\right\|_{L^{\infty}}
	\,\leq\,
	\int_{\R^d}
	\left\|\,
	\nabla_{\bx}\,\Delta^{-1}_{\bx}\,
	\left[f^\eps
	\,-\,
	\tau_{\eps\,\bv}^{-1}f^\eps\right](t,\cdot,\bv)
	\,
	\right\|_{L^{\infty}}
	\,\dD\bv
	\,.
	\]
	To estimate $\ds\left(f^\eps
	\,-\,
	\tau_{\eps\,\bv}^{-1}f^\eps\right)(t,\cdot,\bv)$, we notice that for each $\bv \in \R^d$, it holds
	\[
	\left\|\,
	\nabla_{\bx}\,\Delta^{-1}_{\bx}\,
	\left[f^\eps
	\,-\,
	\tau_{\eps\,\bv}^{-1}f^\eps\right](t,\cdot,\bv)
	\,
	\right\|_{L^{\infty}}
	\,\leq\,
	\left|\eps\,\bv\right|^{\gamma}
	\left\|\,
	\nabla_{\bx}\,\Delta^{-1}_{\bx}\,
	f^\eps
	(t,\cdot,\bv)
	\,
	\right\|_{\scC^{0,\gamma}}\,,
	\]
	and therefore apply Morrey's inequality \eqref{Morrey inequality} which yields
	\[
	\left\|\,
	\bE^{\eps}-\,
	\bI^\eps\,
	\right\|_{L^{\infty}}
	\,\leq\,
	m_{d,p}\,\eps^{\gamma}\,
	\int_{\R^d}
	\left\|
	f^\eps(t,\cdot,\bv)
	\right\|_{L^{p}}
	\,|\bv|^{\gamma}
	\,\dD\bv
	\,.
	\]
	Then we rewrite the latter inequality as follows 
	\[
	\left\|\,
	\bE^{\eps}-\,
	\bI^\eps\,
	\right\|_{L^{\infty}}
	\,\leq\,
	m_{d,p}\,\eps^{\gamma}\,
	\int_{\R^d}
	\left(\,
	\int_{\K^d}
	\left|
	\frac{
		f^\eps(t,\bx,\bv)
	}{\cM(\bv)}
	\right|^p
	\,\dD\bx
	\,
	\right)^{\frac{1}{p}}
	\,|\bv|^{\gamma}\,
	\cM(\bv)
	\,\dD\bv
	\,.
	\]
	Applying H\"older's inequality to the latter relation, we deduce the result
	\[
	\left\|\,
	\bE^{\eps}-\,
	\bI^\eps\,
	\right\|_{L^{\infty}}
	\,\leq\,
	m_{d,p}\,\eps^{\gamma}\,
	\left(\,
	\int_{\R^d}
	\,|\bv|^{\frac{p\,\gamma}{p-1}}\,
	\cM(\bv)
	\,\dD\bv\,
	\right)^{\frac{p-1}{p}}\,
	\LLL{
		f^\eps}_{p
	}
	\,.
	\]
\end{proof}
We now provide estimates for the macroscopic model \eqref{ddP equation}
\begin{proposition}
	\label{estimee rho}
	Consider a smooth solution $\rho$ to equation \eqref{ddP equation}. For all exponent $p$, lying in $(1,+\infty)$, it holds
	\[
	\left\|\,
	\rho(t)\,
	\right\|_{L^{p}
	}
	\,\leq\,
	\max{
		\left(
		\|\,\rho_0\,\|_{L^{p}}
		\,,\,
		\|\,\rho_i\,\|_{L^{p+1}}^{1-1/p^2}
		\right)
	}\,,
	\quad
	\forall\,t\in\R^+\,.
	\]
	Furthermore, it holds
	\[
	\left\|\,
	\rho(t)\,
	\right\|_{L^{\infty}
	}
	\,\leq\,
	\max{
		\left(
		\|\,\rho_0\,\|_{L^{\infty}}
		\,,\,
		\|\,\rho_i\,\|_{L^{\infty}}
		\right)
	}\,,
	\quad
	\forall\,t\in\R^+\,.
	\]
	where $\rho_i$ is given in \eqref{ddP equation}.
\end{proposition}
We postpone the proof of this result to Appendix \ref{proof:prop:rho} since it is not the main point in our analysis.\\

We turn to the main result of this section, in which we provide estimates in $L^p\left(\cM\right)$ for the solution $f^\eps$ to \eqref{kinetic VPFP}. As a direct consequence, we derive convergence estimates for $\pi^\eps$ and $f^\eps$ respectively towards $\rho$ and the local equilibrium $\rho^\eps\,\cM$.
\begin{proposition}\label{a:priori:f:eps}
	Consider some exponent $p>d$ and set $\ds\gamma\,=\, 1\,-\,d\,/\,p$. Let $\left(f^\eps\right)_{\eps>0}$ be a sequence of smooth solutions to \eqref{kinetic VPFP}  whose initial conditions meet assumption \eqref{hyp f eps 0} and $\rho$ be a smooth solution to \eqref{ddP equation} whose initial condition meets \eqref{hyp rho 0 rho i}. There exists a constant $C$ only depending on exponent $p$ and dimension $d$ such that for all $\eps\leq 1$, and for all $t$ less than $T^\eps$ (where $T^\eps$ is given in Theorem \ref{main:th})
	\begin{enumerate}
		\item\label{estimee norme Lp f eps} it holds
		\[
		\LLL{f^\eps}_p\,\leq\,2
		\left(\LLL{f^\eps_0}_p\,+\,
		\eps^{\gamma\left(\frac{1}{2}\,+\,\frac{1}{p-1}\right)}\,
		C\,m_p\right)
		e^{C_{\rho_0,\rho_i}
			C\,
			t}
		\,,
		\]
		where $m_p$ and $C_{\rho_0,\rho_i}$ are respectively defined in \eqref{hyp f eps 0} and \eqref{hyp rho 0 rho i} ;
		\item\label{estimee norme Lp pi eps rho} it holds
		\begin{equation*}
			\left\|\,
			\pi^\eps
			\,-\,
			\rho\,
			\right\|_{L^{2}
			}(t)
			\,\leq\,
			\eps^{\beta}\,
			C\,
			m_p^{2\,p'}\,
			e^{C_{\rho_0,\rho_i}
				\,
				C\,
				t}
			\,,
		\end{equation*}
		where $\beta\,=\,(p-d)\,/\,(p-1)$ and $p'=p/(p-1)$;
		\item\label{estimee f rho eps M L2}
		it holds
		\begin{equation*}
			\left\|
			f^\eps
			\,-\,
			\rho^\eps\,\cM
			\right\|_{L^2
				\left(
				[0,t]\,,\,
				L^{2}
				\left(
				\cM
				\right)
				\right)
			}
			\,\leq\,
			\eps\,
			C\,
			m_p^2
			\,
			e^{C_{\rho_0,\rho_i}
				C\,
				t}
			\,.
		\end{equation*}
	\end{enumerate}
\end{proposition}

\begin{proof}
	The core of this proof consists in deriving item \eqref{estimee norme Lp f eps} in Proposition \ref{a:priori:f:eps}.
	To do so, we consider some positive $\eps$ and  for $t \in \R^+$ define $u$ as follows
	\[
	u(t)
	\,=\,
	\LLL{
		f^\eps(t)}_{p
	}^2
	\,+\,
	\eps^{-\alpha}\,
	\left\|
	\left(
	\pi^\eps
	\,-\,
	\rho\right)(t)
	\right\|_{L^{p}
	}^2\,,
	\]
	for some positive $\alpha$ which needs to be determined. Our strategy consists in estimating each one of the term composing $u(t)$ separately and then to propose a combination of these two estimates which allow us to close the estimate on $u$. In order to simplify notations, we omit the dependence with respect to $\left(t,\bx,\bv\right)$
	when the context is clear. Furthermore, we denote by $C_{d,p}$ a generic positive constant depending only on exponent $p$ and dimension $d$ in this proof.\\
	
	We start by estimating $\LLL{
		f^\eps(t)}_{p
	}$. To do so, we multiply the first line in \eqref{kinetic VPFP} by 
		$
		\ds
		\left|
		f^\eps/
		\cM
		\right|^{p-1}
		$ and integrate over $\K^d\times\R^d$, this yields
		\begin{align*}
		&\frac{1}{p}\,
		\frac{\dD}{\dD t}\,
		\LLL{
			f^\eps}_{p
		}^p
		\,=\,\\[0.8em]
		&-\frac{1}{\eps}\,
		\int_{\K^d\times\R^d}
		\left[
		\bv\cdot
		\nabla_{\bx}\,
		f^\eps
		\,+\,
		\nabla_{\bv}\cdot\left(
		\bE^\eps\,
		f^\eps
		\,-\,
		\frac{1}{\eps}\,
		\left(\,
		\bv\,
		f^\eps
		\,+\,
		\nabla_{\bv}\,
		f^\eps\,
		\right)\right)
		\right]\left|
		\frac{f^\eps}{\cM}\,\right|^{p-1}
		\dD\bx \,\dD \bv\,.
		\end{align*}
		To estimate the latter integral, we first point out that the contribution of the free transport operator cancels since we have
		\[
		\int_{\K^d\times\R^d}
		\bv\cdot
		\nabla_{\bx}\,
		f^\eps
		\left|
		\frac{f^\eps}{\cM}\,\right|^{p-1}
		\dD\bx \,\dD \bv\,=\,
		\frac{1}{p}
		\int_{\K^d\times\R^d}
		\nabla_{\bx}\cdot
		\left(
		\frac{\left|f^\eps\right|^{p}}{\cM^{p-1}}\,\bv\right)
		\dD\bx \,\dD \bv\,=\,
		0\,.
		\]
		According to the latter observation, we deduce that the time derivative of $\LLL{
			f^\eps}_{p
		}$ verifies
		\[
		\frac{1}{p}\,
		\frac{\dD}{\dD t}\,
		\LLL{
			f^\eps}_{p
		}^p
		\,+\,
		\frac{1}{\eps}\,
		\int_{\K^d\times\R^d}
		\nabla_{\bv}\cdot
		\left[
		\bE^\eps\,
		f^\eps
		\,-\,
		\frac{1}{\eps}\,
		\cM\,\nabla_{\bv}\left(\frac{f^\eps}{\cM}\right)
		\right]\left|
		\frac{f^\eps}{\cM}\,\right|^{p-1}
		\dD\bx \,\dD \bv\,=\,0\,,
		\]
		where we also used the relation $\bv\,
		f^\eps
		+
		\nabla_{\bv}\,
		f^\eps=\cM\,\nabla_{\bv}\left(f^\eps/\cM\right)$.
		An integration by part with respect to $\bv$ in the integral term of the latter relation yields
		\[
		\frac{1}{p}\,
		\frac{\dD}{\dD t}\,
		\LLL{
			f^\eps}_{p
		}^p\,+\,
		\frac{1}{\eps^2}\,
		\cD_{p}
		\left[\,
		f^\eps
		\,
		\right]
		\,=\,
		\frac{p-1}{\eps}\,
		\int_{\K^d\times\R^d}
		\bE^{\eps}\,\frac{f^\eps}{\cM}
		\,\nabla_{\bv}\left(
		\frac{f^\eps}{\cM}\,\right)\left|\,\frac{f^\eps}{\cM}\,\right|^{p-2}
		\cM
		\,\dD\bx \,\dD \bv\,.
		\]
	Taking the uniform norm of $\bE^\eps$ and applying Young's inequality in the right hand side of the latter relation, we deduce the following differential inequality
	\begin{equation}\label{estimee:f:0}
		\frac{1}{p}\,
		\frac{\dD}{\dD t}\,
		\LLL{
			f^\eps}_{p
		}^p\,+\,
		\frac{1\,-\,\eta}{\eps^2}\,
		\cD_{p}
		\left[\,
		f^\eps
		\,
		\right]
		\,\leq\,
		\frac{p-1}{4\,\eta}\,
		\left\|
		\bE^\eps
		\right\|_{L^{\infty}
		}^2\,
		\LLL{
			f^\eps}_{p
		}^p\,,
	\end{equation}
	for all positive $\eta$. Then we replace $\ds
	\left\|
	\bE^\eps
	\right\|_{L^{\infty}
	}^2
	$ as follows in the latter estimate
	\[
	\left\|
	\bE^\eps
	\right\|_{L^{\infty}
	}^2
	\,\leq\,
	3\,
	\left(\,
	\left\|\,
	\bE^\eps
	-\,
	\bI^\eps\,
	\right\|_{L^{\infty}
	}^2
	\,+\,
	\left\|\,
	\bI^\eps
	-\,
	\bE\,
	\right\|_{L^{\infty}
	}^2
	\,+\,
	\left\|\,
	\bE\,
	\right\|_{L^{\infty}
	}^2
	\,\right),
	\]
	and estimate the norm $\ds \left\|\,
	\bE^\eps
	\,-\,
	\bI^\eps\,
	\right\|_{L^{\infty}
	}$ applying Lemma \ref{I-E}, the quantities\newline $\ds \left\|\,
	\bI^\eps
	\,-\,
	\bE\,
	\right\|_{L^{\infty}
	}$, $\ds \left\|\,
	\bE\,
	\right\|_{L^{\infty}
	}$ applying Morrey's inequality \eqref{Morrey inequality}. It yields for all $\eta>0$
	\begin{equation}\label{estimee dt f Lp}
		\frac{1}{p}\,
		\frac{\dD}{\dD t}\,
		\LLL{
			f^\eps}_{p
		}^p\,+\,
		\frac{1-\eta}{\eps^2}\,
		\cD_{p}
		\left[\,
		f^\eps
		\,
		\right]
		\,\leq\,
		\frac{C_{d,p}}{\eta}\,
		\left(
		\,
		\eps^{2\,\gamma}\,
		\LLL{
			f^\eps}_{p
		}^2
		\,+\,
		\left\|\,
		\pi^\eps
		\,-\,
		\rho\,
		\right\|_{L^{p}
		}^2
		\,+\,
		\left\|\,
		\rho\,
		\right\|_{L^{p}
		}^2
		\,
		\right)
		\LLL{
			f^\eps}_{p
		}^p\,.
	\end{equation}

	We now estimate $\ds\left\|
	\left(
	\pi^\eps
	-\,
	\rho\right)(t)
	\right\|_{L^{p}}$. To do so, we multiply the difference between equation \eqref{eq pi eps} and the first line of \eqref{ddP equation} by 
		$
		\ds
		\left(\,
		\pi^\eps-\,\rho\,
		\right)
		\left|\,
		\pi^\eps-\,\rho\,
		\right|^{p-2}
		$
		and integrate with respect to $\by$ over $\K^d$, this yields
		\begin{align*}
		&\frac{1}{p}
		\frac{\dD}{\dD t}\,
		\left\|\,
		\pi^\eps
		-\,
		\rho\,
		\right\|_{L^{p}
		}^p\,=\,\\[0.8em]
		&-\int_{\K^d}
		\nabla_{\by}\cdot
		\left[\,
		\int_{\R^d}
		\tau^{-1}_{\eps\bv}
		\left(\bE^\eps f^\eps\right)
		\dD\bv
		-\bE\rho
		-
		\nabla_{\by}\,
		\left(
		\pi^\eps
		-\,
		\rho
		\right)\right]
		\left(
		\pi^\eps-\,\rho\,
		\right)
		\left|
		\pi^\eps-\,\rho
		\right|^{p-2}\dD \by\,,
		\end{align*}
		then we integrate by part with respect to $\by$ in the latter integral term and obtain
		\[
		\frac{1}{p}\,
		\frac{\dD}{\dD t}\,
		\left\|\,
		\pi^\eps
		-\,
		\rho\,
		\right\|_{L^{p}
		}^p\,+\,
		\Delta_{p}
		\left[\,
		\pi^\eps
		-\,
		\rho\,
		\right]
		\,=\,\cA\,,
		\]
		where $\cA$ is given by
		\[
		\cA\,=\,
		(p-1)
		\int_{\K^d}
		\nabla_{\by}\left(
		\pi^\eps
		-
		\rho\right)\cdot
		\left(
		\int_{\R^d}
		\tau_{\eps\,\bv}^{-1}\left(\bE^\eps f^\eps\right)\dD \bv -\,
		\bE\,\rho\right)
		\left|
		\pi^\eps
		-\,
		\rho\right|^{p-2}
		\dD\by\,.
		\]
		To estimate $\cA$, we use the following decomposition
		\[
		\cA\,=\,\cA_1\,+\,\cA_2\,+\,\cA_3\,,
		\]
	where $\cA_1$,
	$\cA_2$
	and
	$\cA_3$ are given by
	\begin{equation*}
		\left\{
		\begin{array}{l}
			\displaystyle
			\ds\cA_1
			\,=\,
			(p-1)
			\int_{\K^d}
			\nabla_{\by}\left(
			\pi^\eps
			-
			\rho\right)\cdot
			\int_{\R^d}
			\left(
			\tau_{\eps\,\bv}^{-1}\,\bE^\eps
			-\,
			\bE^\eps
			\right)g^\eps\,\dD \bv
			\left|
			\pi^\eps
			-\,
			\rho\right|^{p-2}
			\dD\by\,,\\[1.5em]
			\displaystyle
			\ds\cA_2
			\,=\,
			(p-1)
			\int_{\K^d}
			\nabla_{\by}\left(
			\pi^\eps
			-\,
			\rho\right)\cdot
			\left(
			\bE^\eps
			-\,
			\bI^\eps
			\right)\pi^\eps\,
			\left|
			\pi^\eps
			-\,
			\rho\right|^{p-2}
			\dD\by\,,\\[1.5em]
			\displaystyle
			\ds\cA_3
			\,=\,
			(p-1)
			\int_{\K^d}
			\nabla_{\by}\left(
			\pi^\eps
			-\,
			\rho\right)\cdot
			\left(
			\bI^\eps\,\pi^\eps
			-\,
			\bE\,\rho
			\right)
			\left|
			\pi^\eps
			-\,
			\rho\right|^{p-2}
			\dD\by\,.
		\end{array}
		\right.
	\end{equation*}
	$\cA_1$ and $\cA_2$ are error terms which we estimate using Sobolev inequalities whereas we estimate $\cA_3$ using the properties of the limiting macroscopic equation \eqref{ddP equation}.\\
	
	Let us start with $\cA_1$, which we estimate using Young inequality
	\[
	\cA_1
	\,\leq\,
	\eta\,
	\Delta_{p}
	\left[\,
	\pi^\eps
	-\,
	\rho\,
	\right]
	\,+\,
	\frac{p-1}{4\,\eta}\,\cA_{11}\,,
	\]
	for all positive $\eta$ and where $\cA_{11}$ is given by
	\[
	\cA_{11}\,=\,
	\int_{\K^d}
	\left|
	\int_{\R^d}
	\left(
	\tau_{\eps\,\bv}^{-1}\,\bE^\eps
	-\,
	\bE^\eps\right)g^\eps\,\dD \bv\right|^2
	\left|
	\pi^\eps
	-\,
	\rho\right|^{p-2}
	\dD\by\,.
	\]
	Thanks to Morrey inequality \eqref{Morrey inequality}, we have for all $\bv \in \R^d$
	\[
	\left\|
	\tau_{\eps\,\bv}^{-1}\,\bE^\eps
	-\,
	\bE^\eps
	\right\|_{L^{\infty}}
	\,\leq\,
	C_{d,p}\left|\eps\,\bv\right|^{\gamma}\,
	\left\|
	\rho^\eps
	\right\|_{L^{p}}\,.
	\]
	After applying Jensen's inequality to estimate $\left\|
	\rho^\eps
	\right\|_{L^{p}}$, this yields
	\[
	\left\|
	\tau_{\eps\,\bv}^{-1}\,\bE^\eps
	-\,
	\bE^\eps
	\right\|_{L^{\infty}}
	\,\leq\,
	C_{d,p}\left|\eps\,\bv\right|^{\gamma}\,
	\LLL{
		f^\eps}_{p
	}\,.
	\]
	We substitute $\tau_{\eps\,\bv}^{-1}\,\bE^\eps
	-\,
	\bE^\eps$ with the latter estimate in the definition of $\cA_{11}$ and deduce
	\[
	\cA_{11}\,\leq\,C_{d,p}\,\eps^{2\gamma}\,
	\LLL{
		f^\eps}_{p
	}^2\,
	\int_{\K^d}
	\left|
	\int_{\R^d}
	|\bv|^{\gamma}\,g^\eps\,\dD \bv\right|^2
	\left|
	\pi^\eps
	-\,
	\rho\right|^{p-2}
	\dD\by\,.
	\]
	Applying H\"older's inequality, we obtain
	\[
	\cA_{11}\,\leq\,C_{d,p}\,\eps^{2\gamma}\,
	\LLL{
		f^\eps}_{p
	}^2\,
	\left\|\,
	\pi^\eps
	-\,
	\rho\,
	\right\|_{L^{p}
	}^{p-2}\,
	\left(
	\int_{\K^d}
	\left|
	\int_{\R^d}
	|\bv|^{\gamma}\,g^\eps\,\dD \bv\right|^p\,
	\dD\by
	\right)^{2/p}\,.
	\]
	To estimate the integral in the latter inequality, we apply H\"older inequality
	\[
	\left(
	\int_{\K^d}
	\left|
	\int_{\R^d}
	|\bv|^{\gamma}\,g^\eps\,\dD \bv\right|^p\,
	\dD\by
	\right)^{2/p}
	\,\leq\,
	\left(
	\int_{\R^d}
	|\bv|^{p'\gamma}\,\cM\,\dD \bv
	\right)^{2/p'}
	\LLL{
		g^\eps}_{p
	}^2
	\,.
	\]
	Then we notice that $\LLL{
		g^\eps}_{p
	}\,=\,\LLL{
		f^\eps}_{p
	}$ and deduce
	\[
	\cA_{11}\,\leq\,C_{d,p}\,\eps^{2\gamma}\,
	\LLL{
		f^\eps}_{p
	}^4\,
	\left\|\,
	\pi^\eps
	-\,
	\rho\,
	\right\|_{L^{p}
	}^{p-2}\,.
	\]
	In the end, it yields the following bound for $\cA_1$
	\[
	\cA_1
	\,\leq\,
	\eta\,
	\Delta_{p}
	\left[\,
	\pi^\eps
	-\,
	\rho\,
	\right]
	\,+\,
	\frac{C_{d,p}}{\eta}
	\,\eps^{2\gamma}\,
	\LLL{
		f^\eps}_{p
	}^4\,
	\left\|\,
	\pi^\eps
	-\,
	\rho\,
	\right\|_{L^{p}
	}^{p-2}
	\,.
	\]
	
	To estimate $\cA_{2}$, we follow the same method as before excepted that we apply Lemma \ref{I-E} to bound $\bE^\eps-\bI^\eps$, it yields
	\[
	\cA_{2}
	\,\leq\,
	\eta\,
	\Delta_{p}
	\left[\,
	\pi^\eps
	-\,
	\rho\,
	\right]
	\,+\,
	\frac{C_{d,p}}{\eta}
	\,\eps^{2\gamma}\,
	\LLL{
		f^\eps}_{p
	}^4\,
	\left\|\,
	\pi^\eps
	-\,
	\rho\,
	\right\|_{L^{p}
	}^{p-2}
	\,.
	\]
	
	We turn to the last term $\cA_{3}$, which decomposes as follows
	\[
	\cA_{3}
	\,=\,
	\cA_{31}
	\,+\,
	\cA_{32}
	\,+\,
	\cA_{33}
	\,,
	\]
	where the latter terms are given by
	\begin{equation*}
		\left\{
		\begin{array}{l}
			\displaystyle
			\ds
			\cA_{31}
			\,=\,
			(p-1)
			\int_{\K^d}
			\nabla_{\by}\left(
			\pi^\eps
			-\,
			\rho\right)\cdot
			\left(
			\bI^\eps
			-\,
			\bE
			\right)
			\left(
			\pi^\eps
			-\,
			\rho
			\right)
			\left|
			\pi^\eps
			-\,
			\rho\right|^{p-2}
			\dD\by\,,\\[1.5em]
			\displaystyle
			\ds
			\cA_{32}
			\,=\,
			(p-1)
			\int_{\K^d}
			\nabla_{\by}\left(
			\pi^\eps
			-\,
			\rho\right)\cdot
			\left(
			\bI^\eps
			-\,
			\bE
			\right)
			\rho\,
			\left|
			\pi^\eps
			-\,
			\rho\right|^{p-2}
			\dD\by\,,\\[1.5em]
			\displaystyle
			\ds
			\cA_{33}
			\,=\,
			
			(p-1)
			\int_{\K^d}
			\nabla_{\by}\left(
			\pi^\eps
			-\,
			\rho\right)\cdot
			\bE\,
			\left(
			\pi^\eps
			-\,
			\rho
			\right)
			\left|
			\pi^\eps
			-\,
			\rho\right|^{p-2}
			\dD\by\,.
		\end{array}
		\right.
	\end{equation*}
	We start with $\cA_{31}$, which rewrites as follows after an integration by part
	\[
	\cA_{31}
	\,=\,
	-\,\frac{
		p-1}{p}\,
	\int_{\K^d}
	\left|
	\pi^\eps
	-\,
	\rho\right|^p\,\nabla_{\bx}\cdot \left(
	\bI^\eps
	\,-\,
	\bE
	\right)
	\dD\by\,.
	\]
	Therefore, replacing $\bI^\eps$ and $\bE$ according to equations \eqref{def I eps} and \eqref{ddP equation}, we deduce the following relation 
	\[
	\cA_{31}
	\,=\,
	-\,\frac{
		p-1}{p}\,
	\int_{\K^d}
	\left|
	\pi^\eps
	-\,
	\rho\right|^p\left(
	\pi^\eps
	-\,
	\rho
	\right)
	\dD\by\,.
	\]
	Since $\pi^\eps$ has positive values and taking the absolute value of $\rho$, we obtain
	\[
	\cA_{31}
	\,\leq\,
	\frac{
		p-1}{p}\,
	\|
	\rho
	\|_{L^{\infty}}
	\left\|\,
	\pi^\eps
	-\,
	\rho\,
	\right\|_{L^{p}
	}^{p}
	\,.
	\]
	\begin{remark}
		It is the only time that we use the $L^\infty$-norm of $\rho$ in our analysis.
	\end{remark}
	To estimate $\cA_{32}$ and  $\cA_{33}$, we use the same techniques as the ones already used to estimate $\cA_1$ and $\cA_2$. Therefore, we do not detail the computations. In the end it yields
	\[
	\cA_{32}\,+\,\cA_{33}
	\,\leq\,
	\eta\,
	\Delta_{p}
	\left[\,
	\pi^\eps
	-\,
	\rho\,
	\right]
	\,+\,
	\frac{C_{d,p}}{\eta}
	\,
	\|
	\rho
	\|_{L^{p}}^2
	\left\|\,
	\pi^\eps
	-\,
	\rho\,
	\right\|_{L^{p}
	}^{p}
	\,.
	\]
	Gathering latter computations and taking $\eta$ small enough, we obtain the following estimate
	\begin{align*}
		\frac{1}{p}\,
		\frac{\dD}{\dD t}\,
		\left\|\,
		\pi^\eps
		-
		\rho\,
		\right\|_{L^{p}
		}^p
		\leq C_{d,p}
		\left(
		\,\eps^{2\gamma}\,
		\LLL{
			f^\eps}_{p
		}^4
		\left\|\,
		\pi^\eps
		-
		\rho\,
		\right\|_{L^{p}
		}^{p-2}
		+
		\left(
		\|
		\rho
		\|_{L^{p}}^2
		+
		\|
		\rho
		\|_{L^{\infty}}\right)
		\left\|\,
		\pi^\eps
		-
		\rho\,
		\right\|_{L^{p}
		}^{p}\right)
		,
	\end{align*}
	for some constant $C_{d,p}$ only depending on $d$ and $p$. Dividing by $\ds \left\|\,
	\pi^\eps
	\,-\,
	\rho\,
	\right\|_{L^{p}
	}^{p-2}$ the latter estimate, this yields
	\begin{equation}\label{estimee dt pi rho Lp}
		\frac{1}{2}\,
		\frac{\dD}{\dD t}\,
		\left\|\,
		\pi^\eps
		-\,
		\rho\,
		\right\|_{L^{p}
		}^2
		\,\leq\,
		C_{d,p}
		\left(
		\,\eps^{2\gamma}\,
		\LLL{
			f^\eps}_{p
		}^4
		\,+\,
		\left(
		\|
		\rho
		\|_{L^{p}}^2
		\,+\,
		\|
		\rho
		\|_{L^{\infty}}\right)
		\left\|\,
		\pi^\eps
		-\,
		\rho\,
		\right\|_{L^{p}
		}^{2}\right)
		\,.
	\end{equation}
	\\
	
	It is now possible to obtain item \eqref{estimee norme Lp f eps} in Proposition \ref{a:priori:f:eps}: we set $\eta$ to $1$ in \eqref{estimee dt f Lp} and take the sum between estimate \eqref{estimee dt f Lp} divided by $\ds \left\|
	f^\eps
	\right\|_{p
	}^{p-2}$ and estimate \eqref{estimee dt pi rho Lp}, we deduce that $u$ verifies the following differential inequality
	\begin{align*}
	&\frac{1}{2}\,
	\frac{\dD}{\dD t}\,
	u(t)
	\,\leq\,\\[0.8em]
	&C_{d,p}\,
	\left(
	\left(
	\eps^{2\gamma-\alpha}
	+
	\eps^{2\gamma}\right)
	\LLL{
		f^\eps}_{p
	}^4
	\,+\,
	\left\|
	\pi^\eps
	-\,
	\rho
	\right\|_{L^{p}
	}^2
	\LLL{
		f^\eps}_{p
	}^2
	\,+\,
	C_{\rho}
	\left(
	\eps^{-\alpha}
	\left\|
	\pi^\eps
	-\,
	\rho
	\right\|_{L^{p}
	}^{2}
	+
	\LLL{
		f^\eps}_{p
	}^2\right)
	\right)
	,
	\end{align*}
	where $C_{\rho}$ is given by
	\[
	C_{\rho}\,=\,
	\|
	\rho
	\|_{L^{p}}^2
	\,+\,
	\|
	\rho
	\|_{L^{\infty}}\,.
	\]
	Hence, taking $\alpha\,=\,\gamma$ and applying Lemma \ref{estimee rho} to estimate $C_{\rho}$, we deduce 
	\[
	\frac{\dD}{\dD t}\,
	u(t)
	\,\leq\,
	C_{d,p}\,
	\left(
	\eps^{\gamma}\,
	u(t)^2
	\,+\,
	C_{\rho_0,\rho_i}\,u(t)
	\right)
	\,,
	\]
	where $C_{\rho_0,\rho_i}$ is given by
	\[
	C_{\rho_0,\rho_i}\,=\,
	\max{
		\left(
		\|\,\rho_0\,\|_{L^{p}}^2
		,
		\|\,\rho_i\,\|_{L^{p+1}}^{2-2/p^2}
		,
		\|\,\rho_0\,\|_{L^{\infty}}
		\,,\,
		\|\,\rho_i\,\|_{L^{\infty}}
		\right)
	}
	\,.
	\]
	We divide the latter estimate by 
	$
	\ds
	\left(
	\eps^{\gamma}\,
	u(t)^2
	\,+\,
	C_{\rho_0,\rho_i}\,u(t)
	\right)
	$ and notice that
	\[
	\frac{1}{\eps^{\gamma}\,
		u(t)^2
		+
		C_{\rho_0,\rho_i}\,u(t)}
	\,=\,
	\frac{1}{C_{\rho_0,\rho_i}}
	\left(
	\frac{1}{u(t)}
	\,-\,
	\frac{\eps^{\gamma}}{\eps^{\gamma}\,
		u(t)
		+
		C_{\rho_0,\rho_i}}
	\right),
	\]
	therefore, we obtain
	\[
	\frac{\dD}{\dD t}\,
	\ln{
		\left(
		\frac{u(t)}{
			\eps^{\gamma}\,
			u(t)
			\,+\,
			C_{\rho_0,\rho_i}
		}\right)
	}
	\,\leq\,
	C_{\rho_0,\rho_i}\,
	C_{d,p}
	\,.
	\]
	Integrating between $0$ and $t$ and taking the exponential of the latter estimate, it yields
	\begin{equation}\label{estimate:u:final}
		u(t)\,\leq\,
		u(0)
		\left(
		1\,-\,
		\eps^{\gamma}\,
		\frac{u(0)}{C_{\rho_0,\rho_i}}
		\left(
		e^{C_{\rho_0,\rho_i}
			\,
			C_{d,p}\,
			t}
		\,-\,
		1
		\right)
		\right)^{-1}
		e^{C_{\rho_0,\rho_i}
			C_{d,p}\,
			t}
		\,,
	\end{equation}
	for all time $t$ verifying
	\[
	t\,<\,
	\frac{1}{C_{\rho_0,\rho_i}
		\,
		C_{d,p}}\,
	\ln{
		\left(
		1\,+\,
		\frac{C_{\rho_0,\rho_i}}{
			\,u(0)
		}\,
		\eps^{-\gamma}
		\right)
	}\,.
	\]
	To conclude this step, we estimate $u(0)$ by applying the triangular inequality 
	\begin{equation}\label{estime:u0}
		u(0)\,\leq\,
		\LLL{f^\eps_0}_p^2
		\,+\,2\,\eps^{-\gamma}\left(\left\|
		\pi^\eps_0
		\,-\,
		\rho_0^\eps
		\right\|_{L^{p}}^2\,+\left\|
		\rho^\eps_0
		\,-\,
		\rho_0
		\right\|_{L^{p}}^2\right).
	\end{equation}
	Thanks to assumption \eqref{hyp f eps 0}, we obtain
	\[
	u(0)\,\leq\,
	\LLL{f^\eps_0}_p^2\,+\,
	C_{d,p}\,m_p^2\,
	\eps^{2\beta-\gamma}\,,
	\]
	which yields
	\[
	u(0)\,\leq\,
	\LLL{f^\eps_0}_p^2\,+\,
	C_{d,p}\,m_p^2\,
	\eps^{\gamma\left(1\,+\,\frac{2}{p-1}\right)}\,,
	\]
	thanks to the relation $2\beta -\gamma = \gamma(1+2/(p-1))$. Replacing $u(0)$ in \eqref{estimate:u:final} thanks to the latter inequality, we deduce item \eqref{estimee norme Lp f eps} of Proposition \ref{a:priori:f:eps}, that is
	\[
	\LLL{f^\eps}_p\,\leq\,2
	\left(\LLL{f^\eps_0}_p\,+\,
	\eps^{\gamma\left(\frac{1}{2}\,+\,\frac{1}{p-1}\right)}\,
	C_{d,p}\,m_p\right)
	e^{C_{\rho_0,\rho_i}
		C_{d,p}\,
		t}
	\,,
	\]
	for all time $t$ less than $T^\eps$, where $T^\eps$ is given by
	\[
	T^\eps\,=\,
	\frac{1}{C_{\rho_0,\rho_i}
		\,
		C_{d,p}}\,
	\ln{
		\left(
		1\,+\,
		\frac{C_{\rho_0,\rho_i}}{
			4\left(
			\LLL{f^\eps_0}_p^2+
			\eps^{\gamma\left(1\,+\,\frac{2}{p-1}\right)}
			C_{d,p}m_p^2\right)
		}\,
		\eps^{-\gamma}
		\right)
	}\,.
	\]
	
	In order to prove item \eqref{estimee norme Lp pi eps rho}, we consider relation \eqref{estimee dt pi rho Lp}, replace $\ds \LLL{
		f^\eps(t)}_{p
	}$ with the estimate given by item \eqref{estimee norme Lp f eps} in Proposition \ref{a:priori:f:eps} and apply Lemma \ref{estimee rho} to bound $\rho$, this yields
	\begin{equation*}
		\frac{\dD}{\dD t}\,
		\left\|\,
		\pi^\eps
		\,-\,
		\rho\,
		\right\|_{L^{p}
		}^2
		\,\leq\,
		C_{d,p}
		\left(
		\,\eps^{2\gamma}\,
		m_p^4\,e^{C_{\rho_0,\rho_i}
			C_{d,p}\,
			t}
		\,+\,
		C_{\rho_0,\rho_i}
		\,
		\left\|\,
		\pi^\eps
		-\,
		\rho\,
		\right\|_{L^{p}
		}^{2}\right)
		\,,
	\end{equation*}
	for all time $t$ less than $T^\eps$, where $m_p$ and $C_{\rho_0,\rho_i}$ are given in Theorem \ref{main:th}. Multiplying the latter estimate by $\ds e^{-C_{\rho_0,\rho_i}
		C_{d,p}\,
		t}$, integrating between $0$ and $t$ and taking the square root on both sides of the inequality, we obtain
	\begin{equation*}
		\left\|\,
		\pi^\eps
		\,-\,
		\rho\,
		\right\|_{L^{p}
		}
		\,\leq\,
		e^{C_{\rho_0,\rho_i}
			\,
			C_{d,p}\,
			t}\,
		\left(
		\left\|\,
		\pi^\eps_0
		-\,
		\rho_0\,
		\right\|_{L^{p}
		}
		\,+\,
		\eps^{\gamma}\,
		C_{d,p}
		\,m_p^2\,
		t\right)
		\,.
	\end{equation*}
	According to assumption \eqref{hyp f eps 0}, we have
	$\left\|\,
	\pi^\eps_0
	-\,
	\rho_0\,
	\right\|_{L^{p}
	}\,\leq\,m_p\,\eps^{\beta}$. Since $1\,\leq\,m_p $, we deduce 
	\begin{equation*}
		\left\|\,
		\pi^\eps
		-\,
		\rho\,
		\right\|_{L^{p}
		}
		\,\leq\,
		\eps^{\gamma}\,
		C_{d,p}\,
		m_p^2\,
		e^{C_{\rho_0,\rho_i}
			\,
			C_{d,p}\,
			t}
		\,.
	\end{equation*}
	Item \eqref{estimee norme Lp pi eps rho} in Proposition \ref{a:priori:f:eps} is obtained applying the following interpolation inequality in the latter estimate
	\[
	\left\|\,
	\pi^\eps
	-\,
	\rho\,
	\right\|_{L^{2}
	}
	\,\leq\,
	\left\|\,
	\pi^\eps
	-\,
	\rho\,
	\right\|_{L^{1}
	}^{\frac{p-2}{p-1}}
	\left\|\,
	\pi^\eps
	-\,
	\rho\,
	\right\|_{L^{p}
	}^{p'}\,,
	\]
	where $p'$ is given by $p'=p/(p-1)$ and noticing that $p'\,\gamma\,=\,\beta$.
	\\
	
	To prove item \eqref{estimee f rho eps M L2} in Proposition \ref{a:priori:f:eps}, we set $p=2$ in \eqref{estimee:f:0}, apply Morrey's inequality \eqref{Morrey inequality} to estimate $\bE^\eps$ and Jensen's inequality, which ensures $
	\LLL{
		f^\eps}_{2
	}\,\leq\,
	\LLL{
		f^\eps}_{p
	}$. This yields
	\begin{equation*}
		\frac{1}{2}\,
		\frac{\dD}{\dD t}\,
		\LLL{
			f^\eps}_{2
		}^2\,+\,
		\frac{1-\eta}{\eps^2}\,
		\cD_{2}
		\left[\,
		f^\eps
		\,
		\right]
		\,\leq\,
		\frac{C_{d,p}}{\eta}\,
		\LLL{
			f^\eps}_{p
		}^{4}\,.
	\end{equation*}
	Therefore, taking $\eta=1/2$, integrating the latter estimate between $0$ and $t$ and replacing $\LLL{
		f^\eps}_{p
	}$ by the estimate given in item \eqref{estimee norme Lp f eps} in Proposition \ref{a:priori:f:eps} we obtain
	\begin{equation*}
		\int_{0}^t
		\cD_{2}
		\left[\,
		f^\eps
		\,
		\right]\,\dD s
		\,\leq\,
		\eps^2\,
		\left(
		\LLL{
			f^\eps_0}_{p
		}^2
		\,+\,
		C_{d,p}\,
		\int_{0}^t
		m_p^4
		\,
		e^{C_{\rho_0,\rho_i}
			C_{d,p}\,
			s}\,\dD s
		\right)
		\,,
	\end{equation*}
	for all time $t$ between $0$ and $T^\eps$.
	Hence, applying the Gaussian-Poincar\'e inequality which ensures $\ds 
	\LLL{
		f^\eps
		\,-\,
		\rho^\eps\,\cM
	}_{2
	}^2
	\,\leq\,
	\cD_{2}
	\left[\,
	f^\eps
	\,
	\right]
	$,
	we deduce
	\begin{equation*}
		\int_{0}^t
		\LLL{
			f^\eps
			\,-\,
			\rho^\eps\,\cM
		}_{2
		}^2\,\dD s
		\,\leq\,
		C_{d,p}\,\eps^2\,
		m_p^4
		\,
		e^{C_{\rho_0,\rho_i}
			C_{d,p}\,
			t}
		\,,
	\end{equation*}
	which yields the result by taking the square root of the latter estimate.
\end{proof}
Building on Proposition \ref{a:priori:f:eps}, we are able to prove equicontinuity estimates for $f^\eps$ in $L^p\left(\cM\right)$.
\begin{proposition}\label{equicontinuity f eps}
	Consider some exponent $p$ such that $p>d$ and set $\ds \beta \,=\,(p-d)\,/\,(p-1)$. Under assumptions \eqref{hyp rho 0 rho i} on $\ds(\rho_0,\rho_i)$ and  \eqref{hyp f eps 0} on the sequence of initial conditions $(f^\eps_0)_{\eps\,>\,0}$, consider a sequence of solutions $(f^\eps)_{\eps\,>\,0}$ to \eqref{kinetic VPFP}. There exists a constant $C$ only depending on exponent $p$ and dimension $d$ such that for all positive $\eps$ less than $1$, it holds
	\[
	\sup_{|\bx_0|\leq 1}|\bx_0|^{-\beta}
	\LLL{
		f^\eps
		-
		\tau_{\bx_0}\,
		f^\eps
	}_{2
	}(t)
	\,\leq\,
	e^{C_{\rho_0,\rho_i}^{-1}C\, m_p^2\,e^{C_{\rho_0,\rho_i}C\,t}}
	\sup_{|\bx_0|\leq 1}|\bx_0|^{-\beta}
	\LLL{
		f^\eps_0
		-
		\tau_{\bx_0}\,
		f^\eps_0
	}_{p
	}\,.
	\]
	for all time $t$ less than $T^\eps$, where $T^\eps$, $m_p$ and $C_{\rho_0,\rho_i}$ are given in Theorem \ref{main:th}.
\end{proposition}

\begin{proof}
	We consider $\eps>0$ and $(t,\bx_0) \in \R^+\times\K^{d}$ such that $\ds|\bx_0|\,\leq\,1$. Furthermore we denote by $C_{d,p}$ a generic positive constant depending only on exponent $p$ and dimension $d$ in this proof.\\
	We first compute the equation solved by $\ds h^\eps:=f^\eps-\tau_{\bx_0}\,f^\eps$. It is given by the difference between equation \eqref{kinetic VPFP}  and  equation \eqref{kinetic VPFP} translated by $\bx_0$ with respect to the spatial variable, that is
	\[
	\partial_t \, h^\eps
	\,+\,
	\frac{1}{\eps}\,
	\bv\cdot
	\nabla_{\bx}\,
	h^\eps
	\,+\,
	\frac{1}{\eps}
	\left(
	\bE^\eps
	-\tau_{\bx_0}\bE^\eps\right)\cdot
	\nabla_{\bv}\,
	f^\eps
	\,+\,
	\frac{1}{\eps}\,
	\tau_{\bx_0}\bE^\eps\cdot
	\nabla_{\bv}\,
	h^\eps
	\,=\,
	\frac{1}{\eps^2}\,
	\nabla_{\bv}\cdot
	\left[\,
	\bv\,
	h^\eps
	+
	\nabla_{\bv}\,
	h^\eps\,
	\right].
	\]
To estimate the variations of $\LLL{h^\eps}_p$, we proceed as in the estimation of $\LLL{f^\eps}_p$ in the proof of Proposition \ref{a:priori:f:eps}. First, we multiply the latter equation by $\ds\left(h^\eps/\cM\right)\left|h^\eps/\cM\right|^{p-2}$ and integrate over $\K^{d}\times\R^{d}$. Then, we notice that the free transport operator has a zero contribution. Therefore, using the relation $\bv\,
		h^\eps
		+
		\nabla_{\bv}\,
		h^\eps=\cM\,\nabla_{\bv}\left(h^\eps/\cM\right)$ and integrating by part with respect to $\bv$, we obtain
	\begin{align*}
	\frac{1}{p}\,
	\frac{\dD}{\dD t}\,
	\LLL{
		h^\eps}_{p
	}^p&\,+\,
	\frac{1}{\eps^2}\,
	\cD_{p}
	\left[\,
	h^\eps
	\,
	\right]\,
	=\\[0.8em]
	&\frac{p-1}{\eps}
	\int_{\K^d\times\R^d}\left|\frac{h^\eps}{\cM}\right|^{p-2}\nabla_{\bv}\left(
	\frac{h^\eps}{\cM}\,\right)\cdot\left(
	f^\eps
	\left(
	\bE^\eps
	-\tau_{\bx_0}\bE^\eps\right)
	\,+\,
	h^\eps\tau_{\bx_0}\bE^\eps
	\right)
	\dD\bx \,\dD \bv\,.
	\end{align*}
	We apply Young's inequality to estimate the right hand side in the latter inequality and deduce
	\begin{align*}
	\frac{1}{p}\,
	\frac{\dD}{\dD t}\,&
	\LLL{
		h^\eps}_{p
	}^p
	\,\leq\,\\[0.8em]
	&\frac{p-1}{2}
	\int_{\K^d\times\R^d}\left|\frac{h^\eps}{\cM}\right|^{p-2}\left(
	\left|
	\frac{f^\eps}{\cM}\,\right|^2
	\left|
	\bE^\eps
	-\tau_{\bx_0}\bE^\eps\right|^2
	\,+\,
	\left|
	\frac{h^\eps}{\cM}\,\right|^2\left|\tau_{\bx_0}\bE^\eps\right|^2
	\right)\cM\,
	\,\dD\bx \,\dD \bv\,.
	\end{align*}
	After taking the uniform norms of $\ds
	\bE^\eps
	-\tau_{\bx_0}\bE^\eps$ and $\ds
	\bE^\eps$ and applying H\"older's inequality to estimate the cross product between $f^\eps$ and $h^\eps$, it yields 
	\[
	\frac{1}{p}\,
	\frac{\dD}{\dD t}\,
	\LLL{
		h^\eps}_{p
	}^p
	\,\leq\,
	\frac{p-1}{2}
	\left(
	\left\|
	\bE^\eps
	-\tau_{\bx_0}\bE^\eps
	\right\|_{L^{\infty}}^2
	\LLL{h^\eps}_p^{p-2}\LLL{f^\eps}_p^{2}
	\,+\,
	\left\|
	\bE^\eps
	\right\|_{L^{\infty}}^2
	\LLL{h^\eps}_p^p
	\right).
	\]
	We apply Morrey's inequality \eqref{Morrey inequality} to estimate the uniform norms of $\ds
	\bE^\eps
	-\tau_{\bx_0}\bE^\eps$ and $\ds
	\bE^\eps$, it yields \[
	\frac{1}{p}\,
	\frac{\dD}{\dD t}\,
	\LLL{
		h^\eps}_{p
	}^p
	\,\leq\,
	C_{d,p}
	\LLL{h^\eps}_p^{p}\LLL{f^\eps}_p^{2}\,.
	\]
	To conclude, we divide the latter estimate by $\ds \LLL{
		h^\eps}_{p
	}^{p-1}$, multiply it by $\ds e^{-C_{d,p}\int_{0}^{t}\LLL{f^\eps}_p^{2}\,\dD s}$, integrate between $0$ and $t$ and replace $\LLL{f^\eps}_p$ with the estimate in item \eqref{estimee norme Lp f eps} of Proposition \ref{a:priori:f:eps}, it yields
	\[
	\LLL{
		h^\eps}_{p
	}
	\,\leq\,
	\exp{\left(C_{\rho_0,\rho_i}^{-1}
		\left(
		\LLL{
			f^\eps_0}_{p
		}+C_{d,p}m_p\eps^{\frac{\gamma}{2}}
		\right)^2e^{C_{d,p}C_{\rho_0,\rho_i}t}\right)}
	\LLL{
		h^\eps_0}_{p
	}\,,
	\]
	for all time $t$ less than $T^\eps$, where $T^\eps$ is given in Theorem \ref{main:th}.
	We obtain the result dividing the latter estimate by $\ds|\bx_0|^{\beta}$, taking the supremum over all $\bx_0$ with norm less than $1$ and since according to Jensen's inequality it holds $\ds\LLL{
		h^\eps}_{2
	}\,\leq\,\LLL{
		h^\eps}_{p
	}$.
\end{proof}
\section{Proof of Theorem \ref{main:th}}
\label{sec:proof}
We consider the following decomposition of the quantity that we need to estimate
\[
\left\|
f^\eps
\,-\,
\rho\,\cM
\right\|_{L^2
	\left(
	[0,t]\,,\,
	L^{2}
	\left(
	\cM
	\right)
	\right)
}
\,\leq\,
\left(\cE_1\,+\,\cE_2\,+\,\cE_3\right)(t)\,,
\]
where $\cE_1$, $\cE_2$ and $\cE_3$ are given by
\begin{equation*}
	\left\{
	\begin{array}{l}
		\displaystyle
		\ds
		\cE_1(t)
		\,=\,
		\left\|
		f^\eps
		\,-\,
		\rho^\eps\,\cM
		\right\|_{L^2
			\left(
			[0,t]\,,\,
			L^{2}
			\left(
			\cM
			\right)
			\right)
		}\,,\\[1.5em]
		\displaystyle
		\ds
		\cE_2(t)
		\,=\,
		\left\|\rho^\eps
		\,-\,
		\pi^\eps
		\right\|_{L^2
			\left(
			[0,t]\,,\,
			L^{2}
			\left(
			\K^d
			\right)
			\right)
		}\,,\\[1.5em]
		\displaystyle
		\ds
		\cE_3(t)
		\,=\,
		\left\|\pi^\eps
		\,-\,
		\rho
		\right\|_{L^2
			\left(
			[0,t]\,,\,
			L^{2}
			\left(
			\K^d
			\right)
			\right)
		}\,.
	\end{array}
	\right.
\end{equation*}
We estimate $\cE_1$ and $\cE_3$ thanks to Proposition \ref{a:priori:f:eps}. Indeed, according to item \eqref{estimee f rho eps M L2}, it holds
\begin{equation*}
	\cE_1(t)
	\,\leq\,
	C\,
	\eps\,
	m_p^2
	\,
	e^{C_{\rho_0,\rho_i}
		C\,
		t}
	\,,
\end{equation*}
and according to item \eqref{estimee norme Lp pi eps rho}, it holds 
\begin{equation*}
	\cE_3(t)
	\,\leq\,
	C\,\eps^{\beta}\,
	m_p^{2\,p'}\,
	e^{C_{\rho_0,\rho_i}
		\,
		C\,
		t}
	\,,
\end{equation*}
for all time $t$ less than $T^\eps$, where $T^\eps$ is given in  Theorem \ref{main:th}. We turn to the last term $\cE_2$, which we estimate thanks to the following decomposition
\[
\cE_2(t)
\,\leq\,
\cE_{21}(t)\,+\,\cE_{22}(t)\,,
\]
where $\cE_{21}$ and $\cE_{22}$ are defined as follows
\begin{equation*}
	\left\{
	\begin{array}{l}
		\displaystyle  \cE_{21}(t)\,=\,
		\left(
		\int_0^t
		\int_{\K^d}
		\left|
		\int_{\R^{2d}}
		\cM(\Tilde{\bv})
		\left(
		f^\eps(s,\bx,\bv)-f^\eps(s,\bx- \eps \Tilde{\bv}, \bv)\right)\,\dD \Tilde{\bv}\,\dD \bv
		\right|^2\,\dD \bx\,\dD s\right)^{\frac{1}{2}}
		,\\[1.1em]
		\displaystyle  \cE_{22}(t)
		\,=\,\left(\int_0^t
		\int_{\K^d}
		\left|
		\int_{\R^d}
		\cM(\bv)\rho^\eps(s,\by-\eps \bv)-
		f^\eps(s,\by-\eps \bv,\bv)
		\,\dD \bv
		\right|^2\,\dD \by\, \dD s \right)^{\frac{1}{2}}
		\,.
	\end{array}\right.
\end{equation*}
To estimate $\cE_{22}$, we divide and multiply by $\cM(\bv)$ inside the integral in $\bv$ in the definition of $\cE_{22}$ and apply Jensen's inequality, which yields
\[
\cE_{22}(t)
\,\leq\,\left(\int_0^t
\int_{\K^d\times\R^d}
\left|
\frac{\cM(\bv)\rho^\eps(s,\by-\eps \bv)-
	f^\eps(s,\by-\eps \bv,\bv)}{\cM(\bv)}
\right|^2\,\cM(\bv)\,\dD \by\,\dD \bv \,\dD s\right)^{\frac{1}{2}},
\]
then we operate the change of variable $\bx = \by-\eps\,\bv$ in the latter relation and deduce
\[
\cE_{22}(t)
\,\leq\,\cE_1(t)\,.
\]
Thanks to our estimate of $\cE_1(t)$, we obtain
\begin{equation*}
	\cE_{22}(t)
	\,\leq\,
	C\,
	\eps\,
	m_p^2
	\,
	e^{C_{\rho_0,\rho_i}
		C\,
		t}
	\,,
\end{equation*}
Now, we estimate $\cE_{21}$. First, we divide and multiply by $\cM(\bv)$ inside the integral in the definition of $\cE_{21}$ and apply Jensen's inequality, which yields
\[
\cE_{21}(t)
\,\leq\,
\left(
\int_0^t\int_{\R^d}
\LLL{
	f^\eps
	-
	\tau_{\eps\Tilde{\bv}}^{-1}\,
	f^\eps
}_{2
}^2(s)
\cM(\Tilde{\bv})\,\dD \Tilde{\bv}\,\dD s\right)^{\frac{1}{2}}\,.
\]
To bound $\LLL{
	f^\eps
	-
	\tau_{\eps\Tilde{\bv}}^{-1}\,
	f^\eps
}_{2
}$, we distinguish two cases. When, $\ds\left|\eps\,\Tilde{\bv}\right|\,>\,1$ we use the triangular inequality, which ensures
\[
\LLL{
	f^\eps
	-
	\tau_{\eps\Tilde{\bv}}^{-1}\,
	f^\eps
}_{2
}(s)
\,\leq\,2\,|\eps\Tilde{\bv}|^\beta\,
\,\LLL{f^\eps}_2(s)
\,,
\]
whereas in the case $\ds\left|\eps\,\Tilde{\bv}\right|\,\leq\,1$, it holds
\[
\LLL{
	f^\eps
	-
	\tau_{\eps\Tilde{\bv}}^{-1}\,
	f^\eps
}_{2
}(s)
\,\leq\,|\eps\Tilde{\bv}|^\beta
\,\sup_{\left|\bx_0\right|\leq1}\left|\bx_0\right|^{-\beta}
\LLL{
	\tau_{\bx_0}f^\eps-f^\eps
}_2(s)
\,.
\]
According to these estimates, we deduce
\[
\cE_{21}(t)
\,\leq\,
\eps^{\beta}
\left(
\int_0^t 4\,\LLL{f^\eps}_2^2(s)
\,+\,
\left(
\sup_{\left|\bx_0\right|\leq1}\left|\bx_0\right|^{-\beta}
\LLL{
	\tau_{\bx_0}f^\eps-f^\eps
}_2(s)\right)^2\,\dD s\right)^{\frac{1}{2}}\,,
\]
where we used that $ \int_{\R^d} |\Tilde{\bv}|^{2\beta}\cM(\Tilde{\bv})\,\dD \Tilde{\bv}\,\leq\, 1$. Then we apply item \eqref{estimee norme Lp f eps} in Proposition \ref{a:priori:f:eps} to estimate the norm of $f^\eps$ and Proposition \ref{equicontinuity f eps} to estimate the norm of $\tau_{\bx_0}f^\eps-f^\eps$. It yields
\[
\cE_{21}(t)
\,\leq\,
\eps^{\beta}
\left(
\int_0^t C\,m_p^2\,
e^{C_{\rho_0,\rho_i}
	C\,
	s}
\,+\,
m_p^2\exp{\left(C_{\rho_0,\rho_i}^{-1}C\, m_p^2\,e^{C_{\rho_0,\rho_i}C\,s}\right)}\dD s\right)^{\frac{1}{2}}\,,
\]
for all time $t$ less than $T^\eps$, where $T^\eps$ is given in  Theorem \ref{main:th}. Hence, we deduce
\[
\cE_{21}(t)
\,\leq\,\eps^{\beta}
\left(C_{\rho_0,\rho_i}^{-1/2}\,m_p\,
e^{C_{\rho_0,\rho_i}
	C\,
	t}
\,+\,
\exp{\left(C_{\rho_0,\rho_i}^{-1}C\, m_p^2\,e^{C_{\rho_0,\rho_i}C\,t}\right)}
\right)
\,,
\]
where we used the following inequality to estimate the time integral of the double exponential term
\[
m_p^2\,\exp{\left(C_{\rho_0,\rho_i}^{-1}C\, m_p^2\,e^{C_{\rho_0,\rho_i}C\,t}\right)}
\,\leq\,
\frac{1}{C^2}\,
\frac{\dD}{\dD t}
\exp{\left(C_{\rho_0,\rho_i}^{-1}C\, m_p^2\,e^{C_{\rho_0,\rho_i}C\,t}\right)}\,.
\]
Gathering the estimate on $\cE_1$, $\cE_2$ and $\cE_3$, we obtain the estimate in Theorem \ref{main:th}, that is
\[
\left\|
f^\eps
\,-\,
\rho\,\cM
\right\|_{L^2
	\left(
	[0,t]\,,\,
	L^{2}
	\left(
	\cM
	\right)
	\right)
}
\,\leq\,
\eps^{\beta}
\left(
C\,m_p^{2\,p'}\,
e^{C_{\rho_0,\rho_i}
	C\,
	t}
\,+\,
\exp{\left(C_{\rho_0,\rho_i}^{-1}C\, m_p^2\,e^{C_{\rho_0,\rho_i}C\,t}\right)}
\right)
\,,
\]
for all time $t$ less than $T^\eps$, where $T^\eps$ is given in Theorem \ref{a:priori:f:eps}.\\
\section{Conclusion}
We have proposed a method in order to treat the diffusive scaling for the VPFP model. Our approach provides non-perturbative strong convergence results with explicit rates. It may be regarded as an alternative to compactness methods relying on averaging lemmas widely used in this context \cite{Masmoudi_Tayeb07,El_Ghani_Masmoudi,Ghani10,Wu_Lin_Liu15,Herda16} with the advantage that it provides explicit convergence rates. \\

An interesting and challenging continuation of this work would consist in conciliating this approach with hypocoercivity methods \cite{Herau_Thomann16,Herda_Rodrigues, Lanoir/Dolbeault/Xingyu/M.Lazhar} which present the advantage of providing global in time convergence estimates but which fail, for now, to provide non-perturbative results. To be noted that up to our knowledge, non-perturbative results \cite{Bouchut_Dolbeault95,Bonilla_Carrillo_Soler97}  treating the long time behavior of \eqref{kinetic VPFP} rely on compactness arguments and are thus non quantitative. Therefore, associating hypocoercivity methods with the one presented in this article might be a way to treat simultaneously the diffusive regime $\eps\rightarrow 0$ and the long time behavior $t\rightarrow+\infty$ in a non-perturbative framework and with explicit rates. \\

Another natural question concerns the applicability of our approach to treat other asymptotic limits of the VPFP model. For example, our method might be applicable to "free-field" regimes analyzed by M. Herda and M. Rodrigues in \cite{Herda_Rodrigues}. These regimes correspond to the limit $\tau\rightarrow 0$ when $\eps^2$ is replaced with $\tau\eps$ in the right hand of the first line in \eqref{kinetic VPFP} and under the assumption $\tau=o(\eps)$ as $\tau\rightarrow 0$. Roughly speaking, these regimes describe situations where collisions are strong enough to cancel electrostatic effects. \\
Adapting our approach to the famous high-field or hyperbolic regime also constitutes a great challenge. This regime corresponds to a situation where collisions and electrostatic effects have the same magnitude, leading to unthermalized asymptotic limits. This regime has drawn intense interest of the mathematics community \cite{Poupaud92,Cercignani_Gamba_Levermore97,Abdallah_Chaker01,Bonilla_Soler01,Nieto_Poupaud_Soler01,Goudon_Nieto_Poupaud_Soler05,Bostan_Goudon08}. In this case, smoothing effects due to the Fokker-Planck operator disappear in the limit. For this reason, there is no clear indication that our method would apply. However, we also mention the article \cite{Bellouquid_Calvo_Nieto_Soler13} in which intermediate regimes are considered, where collisions slightly dominate electrostatic effects. If possible, applying our method to these intermediate regimes would constitute a first step towards treating the high-field limit. \\

To conclude, it would be interesting to test the robustness of our method on other collision operators such as linearized Boltzmann operators \cite{Poupaud91,Abdallah_Tayeb04,Masmoudi_Tayeb07} or BGK relaxation operators \cite{Cercignani_Gamba_Levermore97}. In these examples, we expect less smoothing effects than in the Fokker-Planck case, leading to additional difficulties.

\appendix

\section{Proof of Proposition \ref{estimee rho}
}\label{proof:prop:rho}
Let us start with the case where exponent $p$ is strictly less than $+\infty$. We compute the time derivative of 
$
\ds
\|
\rho
\|_{L^{p}}^p 
$ by multiplying equation \eqref{ddP equation} by
$
\ds
\rho^{p-1}
$ and integrating with respect to $\bx$. After an integration by part, we obtain 
\[
\frac{1}{p}\,
\frac{\dD}{\dD t}\,
\left\|\,
\rho\,
\right\|_{L^{p}
}^p\,+\,
\Delta_{p}
\left[\,
\rho\,
\right]
\,=\,
\frac{p-1}{p}\,\cD\,,
\]
where $\cD$ is given by
\[
\cD\,=\,
\int_{\K^d}
\bE\cdot
\nabla_{\bx}\rho^{\,p}\,
\dD\bx\,.
\]
Integrating by part and according to equation \eqref{ddP equation}, we rewrite $\cD$ as follows
\[
\cD\,=\,
\int_{\K^d}
\left(
\rho_i\,-\,\rho
\right)\rho^{\,p}\,
\dD\bx\,=\,
\int_{\K^d}
\rho_i\,\rho^{\,p}\,
\dD\bx\,-\,\|\rho\|_{L^{p+1}}^{p+1}
\,.
\]
To estimate the $\cD$, we apply H\"older's inequality
\[
\cD\,\leq\,
\|\rho_i\|_{L^{p+1}}\,
\|\rho\|_{L^{p+1}}^{p}
\,-\,
\|\rho\|_{L^{p+1}}^{p+1}\,,
\]
and therefore deduce
\[
\cD\,\leq\,
\|\rho_i\|_{L^{p+1}}^{p+1}\,
\mathds{1}_{
	\|\rho\|_{L^{p+1}}
	\,\leq\,
	\|\rho_i\|_{L^{p+1}}}\,.
\]
Furthermore, we use that $\rho$ is a probability measure to apply Jensen's inequality and we deduce
\[
\|\rho\|_{L^{p}}
\,\leq\,
\|\rho\|_{L^{p+1}}^{(p^2-1)/p^2}\,.
\]
Injecting this inequality in the latter estimate on $\cD$, we obtain
\[
\cD\,\leq\,
\|\rho_i\|_{L^{p+1}}^{p+1}\,
\mathds{1}_{
	\|\rho\|_{L^{p}}
	\,\leq\,
	\|\rho_i\|_{L^{p+1}}^{(p^2-1)/p^2}}\,.
\]
Therefore, we obtain
\[
\frac{\dD}{\dD t}\,
\left\|\,
\rho\,
\right\|_{L^{p}
}^p
\,\leq\,
(p-1)\,
\|\rho_i\|_{L^{p+1}}^{p+1}\,
\mathds{1}_{
	\|\rho\|_{L^{p}}^p
	\,\leq\,
	\|\rho_i\|_{L^{p+1}}^{(p^2-1)/p}}\,.
\]
One can check that for any positive $\eta$, the constant
\[
C_{\eta}\,=\,
\max{
	\left(
	\|\rho_0\|_{L^{p}}^p
	\,,\,
	\|\rho_i\|_{L^{p+1}}^{(p^2-1)/p}
	\right)
}
\,+\,\eta\,
\]
is a super solution to the latter differential inequality. Therefore, it holds for all $\eta>0$
\[
\left\|\,
\rho\,
\right\|_{L^{p}
}^p
\,\leq\,C_{\eta}\,.
\]
Hence, taking $\eta \rightarrow 0$, we obtain the expected result
\[
\left\|\,
\rho\,
\right\|_{L^{p}
}
\,\leq\,
\max{
	\left(
	\|\rho_0\|_{L^{p}}
	\,,\,
	\|\rho_i\|_{L^{p+1}}^{1-1/p^2}
	\right)
}\,,
\]
for all time $t\,\geq\,0$. The case $p\,=\,+\infty$ is obtained taking the limit $p\,\rightarrow\,+\infty$ in the latter estimate.\\

\bibliographystyle{siamplain}
\bibliography{refer}
\end{document}

%% file: macro.tex
\usepackage{bm}
\usepackage{mathrsfs}
\usepackage{amsfonts,amssymb,oldgerm}
\usepackage{dsfont} 
\usepackage{empheq}
\usepackage{cases}
\usepackage{tikz}

\newcommand\dD{\mathrm{d}}

\newcommand{\LLL}[1]{{\left\vert\kern-0.25ex\left\vert\kern-0.25ex\left\vert #1 
		\right\vert\kern-0.25ex\right\vert\kern-0.25ex\right\vert}}
	
\newcommand\ds{\displaystyle}

\def\eps{\varepsilon }

\newcommand\bx{{\bm x}}
\newcommand\by{{\bm y}}

\newcommand\bv{{\bm v}}
\newcommand\bE{{\mathbf{E}}}
\newcommand\bI{{\mathbf I}}

\newcommand{\R}{{\mathbb R}}
\newcommand{\T}{{\mathbb T}}
\newcommand{\K}{{\mathbb K}}

\newcommand\cA{{\mathcal A}}

\newcommand\cD{{\mathcal D}}
\newcommand\cE{{\mathcal E}}

\newcommand\cM{{\mathcal M}}

\newcommand\scC{{\mathscr C}}

\newcommand{\beq}{\begin{equation}}
\newcommand{\eeq}{\end{equation}}

\newcommand{\be}{\begin{equation}}
\newcommand{\ee}{\end{equation}}
\newcommand{\bl}{\begin{lemma}}
\newcommand{\el}{\end{lemma}}

\definecolor{myblue}{rgb}{0.25,0.50,0.99} 
\definecolor{myred}{rgb}{1.0,0.3,0.2}

%% file: ex_shared.tex
\usepackage{lipsum}
\usepackage{amsfonts}
\usepackage{graphicx}
\usepackage{epstopdf}
\usepackage{algorithmic}
\ifpdf
  \DeclareGraphicsExtensions{.eps,.pdf,.png,.jpg}
\else
  \DeclareGraphicsExtensions{.eps}
\fi

\newsiamremark{remark}{Remark}
\newsiamremark{hypothesis}{Hypothesis}
\crefname{hypothesis}{Hypothesis}{Hypotheses}
\newsiamthm{claim}{Claim}

\headers{Diffusive limit of the Vlasov-Poisson-Fokker-Planck model}{A. Blaustein}

\title{Diffusive limit of the Vlasov-Poisson-Fokker-Planck model:
	quantitative and strong convergence results\thanks{Submitted to the editors October 24, 2022.
\funding{This work was funded by ANR MUFFIN No: ANR-19-CE46-0004.}}}

\author{Alain Blaustein\thanks{Universit\'e Toulouse III Paul Sabatier, 118, route de Narbonne, 31062 Toulouse Cedex 9, France
  (\email{alain.blaustein@math.univ-toulouse.fr}).}
}

\usepackage{amsopn}
